\documentclass[11pt,a4paper]{amsart}

\usepackage{amsmath}
\usepackage{mathtools}
\usepackage{amsthm}
\usepackage{amssymb}
\usepackage{hyperref}

\newcommand*\fullref[3][\relax]{%
  \ifdefined\hyperref%
    {\hyperref[#3]{#2\penalty 200\ \ref*{#3}#1}}%
  \else%
    {#2\penalty 200\ \relax\ref{#3}#1}%
  \fi%
}

\usepackage{tikz}
\usetikzlibrary{matrix}

\tikzset{
  pretableaumatrix/.style={
    ampersand replacement=\&,
    matrix of math nodes,
    outer sep=1mm,
    inner sep=0mm,
    anchor=center,
    row sep={between borders,-\pgflinewidth},
    column sep={between borders,-\pgflinewidth},
    dottedentry/.style={densely dotted},
    dashedentry/.style={densely dashed},
    spaceentry/.style={draw=none,execute at begin node=\null},
  },
  pretableaunode/.style={
    font=\small,
    draw=gray,
    sharp corners,
    rectangle,
    anchor=base,
    text height=3.75mm,
    text depth=1.25mm,
    minimum height=5mm,
    minimum width=5mm,
    inner sep=0mm,
    outer sep=0mm,
    doublewidth/.style={minimum width=10mm},
    footnotesize/.style={font=\footnotesize},
    scriptsize/.style={font=\scriptsize},
  },
  tableaumatrix/.style={
    pretableaumatrix,
    every node/.append style={
      pretableaunode,
    },
  },
  medtableaumatrix/.style={
    pretableaumatrix,
    every node/.append style={
      pretableaunode,
      font=\footnotesize,
      text height=2.75mm,
      text depth=.75mm,
      minimum height=3.5mm,
      minimum width=3.5mm
    },
  },
  smalltableaumatrix/.style={
    pretableaumatrix,
    every node/.append style={
      pretableaunode,
      font=\scriptsize,
      text height=1.85mm,
      text depth=.15mm,
      minimum height=2.5mm,
      minimum width=2.5mm,
    },
  },
  tinytableaumatrix/.style={
    pretableaumatrix,
    every node/.append style={
      pretableaunode,
      font=\tiny,
      text height=1.25mm,
      text depth=.15mm,
      minimum height=1.75mm,
      minimum width=1.75mm
    },
  },
  tableau/.style={
    baseline=-1.25mm,
    every matrix/.style={tableaumatrix},
  },
  medtableau/.style={
    baseline=-1.25mm,
    every matrix/.style={medtableaumatrix},
  },
  smalltableau/.style={
    baseline=-1.25mm,
    every matrix/.style={smalltableaumatrix},
  },
  preshapetableaumatrix/.style={
    pretableaumatrix,
    execute at end cell={\strut},
    every node/.append style={
      draw=black,
      anchor=base,
      inner sep=0mm,
      outer sep=0mm,
    },
    shadedentry/.style={fill=gray},
    darkshadedentry/.style={fill=darkgray},
  },
  medshapetableaumatrix/.style={
    preshapetableaumatrix,
    every node/.append style={
      text height=2.75mm,
      text depth=.75mm,
      minimum height=3.5mm,
      minimum width=3.5mm
    },
  },
  shapetableaumatrix/.style={
    ampersand replacement=\&,
    matrix of math nodes,
    outer sep=0mm,
    inner sep=0mm,
    anchor=base,
    row sep={between borders,-\pgflinewidth},
    column sep={between borders,-\pgflinewidth},
    execute at begin cell={\strut},
    every node/.append style={draw,anchor=base,text height=1mm,text depth=.5mm,minimum size=1.5mm,inner sep=0mm,outer sep=0mm},
  },
  shapetableau/.style={
    every matrix/.style={shapetableaumatrix},
  },
  topalign/.style={
    every matrix/.append style={name=maintableau,anchor=maintableau-1-1.base},
    baseline,
  },
}

\newcommand*\tableau[2][]{\tikz[tableau,#1]\matrix{#2};}

\newcommand*{\textparens}[1]{\textup{(}#1\textup{)}}

\newcommand*{\defterm}[1]{\emph{#1}}

\ifdefined\enquote
  
\else
  
\fi

\makeatletter
\newcommand\chyph{\penalty\@M-\hskip\z@skip}
\newcommand\cendash{\penalty\@M--\hskip\z@skip}
\newcommand\cslash{\penalty\@M/\hskip\z@skip}
\makeatother

\newcommand*\claptowidth[2]{%
    \sbox0{#2}%
    \kern 0.5\wd0\clap{#1}\kern 0.5\wd0\relax%
}
\theoremstyle{definition}
\newtheorem{definition}{Definition}[section]

\newtheorem{problem}[definition]{Open Problem}

\theoremstyle{plain}
\newtheorem{corollary}[definition]{Corollary}
\newtheorem{lemma}[definition]{Lemma}
\newtheorem{proposition}[definition]{Proposition}
\newtheorem{theorem}[definition]{Theorem}

\numberwithin{equation}{section}
\DeclarePairedDelimiter{\floor}{\lfloor}{\rfloor}
\DeclarePairedDelimiter{\parens}{\lparen}{\rparen}

\DeclarePairedDelimiter{\set}{\{}{\}}
\DeclarePairedDelimiterX{\gset}[2]{\{}{\}}{\,#1:#2\,}
\newcommand*\nset{\mathbb{N}}

\newcommand*\pset{\mathbb{P}}

\makeatletter
\newcommand*\powerset{\@ifstar\@powersetnoparens\@powersetparens}
\newcommand*\@powersetparens[2][]{\pset\parens[#1]{#2}}
\newcommand*\@powersetnoparens[1]{\pset{#1}}
\makeatother

\makeatletter
\newcommand*{\biggg}{\bBigg@{4}}

\newcommand*{\Biggg}{\bBigg@{5}}

\newcommand*{\sizeddelimiter}[2]{\bBigg@{#1}#2}
\makeatother

\makeatletter
\newcommand*{\sizedsurd}[2][]{%
  {\@mathmeasure\z@{\nulldelimiterspace\z@}%
     {\sqrt[#1]{\vcenter to #2\big@size{}}}%
     \box\z@}%
 }
\makeatother

\DeclarePairedDelimiterX{\pres}[2]{\langle}{\rangle}{#1\,\delimsize\vert\,\mathopen{}#2}

\newcommand*\mathclaptowidth[2]{%
    \sbox0{$#2$}%
    \kern 0.5\wd0\mathclap{#1}\kern 0.5\wd0\relax%
}
\DeclarePairedDelimiter{\mset}{\langle}{\rangle}

\newcommand*\setyoung[1]{\mathrm{YT}\parens{#1}}
\newcommand*\setminors[3][]{\mathrm{M}_{#2}\parens[#1]{#3}}
\newcommand*\msetminors[3][]{\mathrm{mM}_{#2}\parens[#1]{#3}}

\begin{document}

\title{Reconstructing Young Tableaux}

\author{Alan J. Cain}
\address[A. J. Cain]{%
Centro de Matem\'{a}tica e Aplica\c{c}\~{o}es\\
Faculdade de Ci\^{e}ncias e Tecnologia\\
Universidade Nova de Lisboa\\
2829--516 Caparica\\
Portugal
}
\email{%
a.cain@fct.unl.pt
}
\thanks{This work was partially supported by by the Funda\c{c}\~{a}o para a Ci\^{e}ncia e a
  Tecnologia (Portuguese Foundation for Science and Technology) through the project {\scshape UIDB}/00297/2020
  (Centro de Matem\'{a}tica e Aplica\c{c}\~{o}es) and the project
  {\scshape PTDC}/{\scshape MAT-PUR}/31174/2017.}

\author{Erkko Lehtonen}
\address[E. Lehtonen]{%
Centro de Matem\'{a}tica e Aplica\c{c}\~{o}es\\
Faculdade de Ci\^{e}ncias e Tecnologia\\
Universidade Nova de Lisboa\\
2829--516 Caparica\\
Portugal
}
\email{%
e.lehtonen@fct.unl.pt
}

\begin{abstract}
  This paper completely characterizes the standard Young tableaux that can be reconstructed from their sets or multisets
  of $1$-minors. In particular, any standard Young tableau with at least $5$ entries can be reconstructed from its set
  of $1$-minors.
\end{abstract}


\maketitle

\section{Introduction}

Reconstruction problems are a very general class of problems that ask whether a mathematical object is uniquely
determined by a collection of pieces of partial information about the object.  A classical example of such a problem,
posed by Kelly~\cite{Kelly} and Ulam~\cite{Ulam}, is present in a famous unsolved question in graph theory, the graph
reconstruction conjecture, which concerns whether every finite simple graph with at least two vertices is uniquely
determined, up to isomorphism, by the collection of its one-vertex-deleted induced subgraphs.  Analogous reconstruction
problems have been defined and studied for many kinds of mathematical objects, such as relations, posets, matrices,
matroids, and permutations.

A reconstruction problem for integer partitions was first formulated by Mnukhin~\cite{Mnukhin_properties} and
Cameron~\cite{Cameron_stories} and can be stated as follows: Is a partition of $n$ uniquely determined by its set of
$k$-minors?  Here, a $k$-minor of a partition $\lambda$ of $n$ is a partition of $n - k$ whose Young diagram fits inside
that of $\lambda$.  The problem of determining the values of $n$ and $k$ for which any partition of $n$ can be
reconstructed from its set of $k$-minors has been studied by several authors, and bounds for feasible values of $n$ and
$k$ were obtained, e.g., by Pretzel and Siemons~\cite{PretzelSiemons} and Vatter~\cite{Vatter}.  The exact solution to
the partition reconstruction problem was provided by Monks~\cite{monks_solution}.

% The exact solution to the partition reconstruction problem \dash the question of determining the values of $n$ and $k$
% for which any partition of $n$ can be reconstructed from its set of $k$-minors \dash was solved by Monks
% \cite{monks_solution}. Here, a $k$-minor of a partition $\lambda$ of $n$ is a partition of $n - k$ whose Young diagram
% fits inside that of $\lambda$.

Monks proposed the analogous question for standard Young tableaux, asking which $n$ and
$k$ have the property that any standard Young tableau with $n$ entries can be reconstructed from its set of $k$-minors
\cite[Subsection~4.3]{monks_solution}. A $k$-minor of a standard Young tableau with $n$ elements is a standard Young
tableau with $n-k$ entries obtained in natural way by deleting entries using jeu de taquin
\cite[Subsection~1.2]{fulton_young} and renumbering.

This paper takes the first steps towards answering Monks's question, by completely characterizing the standard Young
tableaux that can be reconstructed from their sets or multisets of $1$-minors. In particular, for $n \geq 5$, any
standard Young tableau with $n$ entries can be reconstructed from its set \textparens{and thus from its multiset} of
$1$-minors \textparens{see \fullref{Theorem}{thm:set-reconst}}. This bound is the best possible, since there are
tableaux with $4$ entries that are not reconstructible from their sets \textparens{or multisets} of $1$-minors
\textparens{see \fullref{Section}{sec:reconst-multiset}}.

\section{Preliminaries}

Let $n \in \nset$. A \defterm{partition} $\lambda$ of $n$ is a non-increasing finite sequence
$(\lambda_1,\ldots,\allowbreak\lambda_m)$ whose terms are in $\nset$ and sum to $n$.

The \defterm{Young diagram} of shape $\lambda$, where $\lambda$ is a partition of $n$, is a left-aligned array of cells,
with $\lambda_h$ boxes in the $h$-th row \textparens{counting from the top}. For example, the Young diagram of shape
$(4,3,1,1)$ is
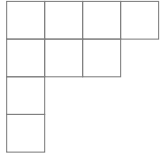
\[
\tableau{
\null \& \null \& \null \& \null \\
\null \& \null \& \null \\
\null \\
\null \\
}\quad.
\]

A \defterm{standard Young tableau} of shape $\lambda$, where $\lambda$ is a partition of $n$, is a Young diagram of
shape $\lambda$ in which every cell contains one of the natural numbers $1,2,\ldots,n$, each appearing exactly once,
such that the entries in each row are increasing from left to right, and the entries in each column are increasing from
top to bottom. For example, a standard Young tableau of shape $(4,3,1,1)$ is
\begin{equation}
\label{eq:youngtableaueg}
\tableau{
1 \& 2 \& 7 \& 8 \\
3 \& 5 \& 9 \\
4 \\
6 \\
}\quad.
\end{equation}
The set of all standard Young tableaux with $n$ entries is denoted $\setyoung{n}$. For brevity, this paper uses
\defterm{tableau} to mean a standard Young tableau. An \defterm{outer corner} of a tableau is a cell that has no cell
immediately below or to the right of it; in \eqref{eq:youngtableaueg}, the entries $6$, $8$, and $9$ occupy outer
corners.

There is a natural way of deleting entries from a tableau using what is known as `jeu de taquin'. Let
$T \in \setyoung{n}$ and $m \leq n$. To delete $m$ from $T$, first remove the cell containing $m$ from $T$, leaving a
space. Iterate the following process until it terminates: Consider the cell $R$ to the right of the space \textparens{if
  such a cell exists} and the cell $B$ below the space \textparens{if such a cell exists}. If the entry in $R$ is
smaller than that in $B$ or $B$ does not exist, slide $R$ into the space, leaving a new space where $R$ was; if the
entry in $B$ is smaller than that in $R$ or $R$ does not exist, slide $B$ into the space, leaving a new space where $B$
was; if neither exists, the space is where an outer corner was previously and the process terminates. Now renumber each
entry $p > m$ to $p-1$. This yields a new tableau in $\setyoung{n-1}$, which is denoted $T - m$. The term \defterm{jeu
  de taquin} specifically refers to the iterative sliding process.

For a tableau $T \in \setyoung{n}$ and $k \leq n$, a $k$-minor can be formed by iteratively deleting $k$ entries from
$T$. The set of $k$-minors of $T$ is denoted $\setminors{k}{T}$; the multiset of $k$-minors of $T$ is denoted
$\msetminors{k}{T}$.

\section{%
  \texorpdfstring
  {Reconstruction from the set of $1$-minors}
  {Reconstruction from the set of 1-minors}%
}

This section is devoted to proving that any tableau with at least $5$ entries can be reconstructed from its set of
$1$-minors. The key idea of the proof is that the set of $1$-minors determines (a) the location of the largest entry in
the tableau and (b) the set of 1-minors of the tableau obtained by deleting the largest entry. This allows a proof by
induction on the number of entries.

We begin with the observation that the shape of any tableau with at least $3$ entries can be reconstructed from the set of $1$-minors of the tableau.
This follows immediately from the results of Monks \cite[Theorem~2.1]{monks_solution}, but for the sake of self-containment, we provide a simple proof of this fact.

\begin{lemma}
  \label{lem:1minor-shape-reconst}
  The set of $1$-minors of a tableau in $\setyoung{n}$ for $n \geq 3$ determines the shape of the tableau.
\end{lemma}

\begin{proof}
  Let $T \in \setyoung{n}$ where $n \geq 3$.

  Suppose first that every tableau in $\setminors{1}{T}$ has the same shape. Then all jeu de taquin processes in $T$
  must end at the same outer corner. Thus $T$ has only one outer corner, which implies that $T$ is rectangular, in the
  sense of having shape $\parens{\ell,\ell,\ldots,\ell}$ for some $\ell$, and the tableaux in $\setminors{1}{T}$ will
  have shape $\parens{\ell,\ell,\ldots,\ell,\ell-1}$. If there are at least two rows and columns in this shape, then the
  shape of $T$ can be obtained by taking the shape of the tableaux in $\setminors{1}{T}$ and adding $1$ to the last
  part. Otherwise the shape of the $1$-minors has only one row or only one column \textparens{but not both since there
    are at least $3$ entries in $T$ and so at least $2$ entries in every $1$-minor}, and the shape of $T$ is obtained by
  extending the row or column by $1$, since this is the only way to add a cell and leave only one outer corner.

  % Since $T$ has a unique outer corner, this entry must be $n$.

  Suppose that $\setminors{1}{T}$ contains tableaux of at least two different shapes. Then $T$ must have at least two outer
  corners, and so the shape of $T$ can be obtained by taking the union of all shapes of tableaux in $\setminors{1}{T}$,
  since an outer corner of $T$ absent from one shape will be present in another.
\end{proof}

\begin{lemma}
  \label{lem:1minor-n-pos-reconst}
  The set of $1$-minors of a tableau in $\setyoung{n}$ for $n \geq 4$ determines the location of the largest entry in
  that tableau.
\end{lemma}

\begin{proof}
  Let $T \in \setyoung{n}$. By \fullref{Lemma}{lem:1minor-shape-reconst}, the shape of $T$ can be determined from
  $\setminors{1}{T}$.

  If $T$ has only one outer corner, this must be the location of $n$.

  For the purposes of this proof, an outer corner of a tableau in $\setminors{1}{T}$ that corresponds to an outer corner in
  $T$ is called a \defterm{surviving} outer corner. If $T$ has more than one outer corner, then since at least one
  jeu de taquin process terminates at each outer corner, each outer corner of $T$ corresponds to a surviving outer
  corner in at least one tableau in $\setminors{1}{T}$.

  Consider how the symbol $n-1$ can appear in a minor in $\setminors{1}{T}$. It is the largest symbol in such a minor, so
  it must be in an outer corner. It either arises from the entry $n$ being renumbered to $n-1$ \textparens{when the
    minor is formed by deleting any entry in $\set{1,\ldots,n-1}$} or else it corresponds to the entry $n-1$ in $T$
  \textparens{when the minor is formed by deleting $n$ and no entries are renumbered}.

  Suppose that $T$ has multiple outer corners. Each of these outer corners corresponds to a surviving outer corner in
  some tableau in $\setminors{1}{T}$. In particular, the cell containing $n$ in $T$ corresponds to a surviving outer
  corner containing $n-1$ in any minor arising from a jeu de taquin process that ends at a different outer corner.  On
  the other hand, $n-1$ can appear in some other surviving outer corner at most once, in the minor arising from the
  deletion of $n$ \textparens{which only happens when $n-1$ is in a different outer corner of $T$}.

  If there are at least three outer corners in $T$, there are at least two minors where $n-1$ appears in the surviving
  outer corner corresponding to the cell containing $n$ in $T$, so in this case the location of $n$ in $T$ is
  determined.

  The case that remains is when $T$ has exactly two outer corners; assume this case. Then there are at least two
  distinct minors. If $n-1$ appears more than once in the same surviving outer corner in at least two different minors,
  the corresponding outer corner of $T$ contains $n$. If an outer corner of $T$ corresponds to a surviving outer corner
  of a minor containing some symbol $k < n-1$, then this outer corner of $T$ cannot contain $n$, for this $k$ is either
  unchanged from $T$ or arises via renumbering from $k+1 < n$, and so the other outer corner of $T$ contains $n$.

  Suppose the jeu de taquin process that arises in deleting $k < n$ ends at the outer corner containing $n$. Then there
  would be a minor in which the other surviving outer corner contains $n-2$, as a result of renumbering $n-1$. By the
  previous paragraph, this determines the location of $n$ in $T$. So assume there is only one jeu de taquin process
  ending at $n$ in $T$. Since $T$ has only two outer corners, $n$ must be either the bottom row or the rightmost column
  in $T$, and every jeu de taquin path starting in this row or column must end at $n$. Thus $n$ is the only entry in
  this row or column. This means $T$ is a rectangle with a single cell at one side or below \textparens{that is, shape
    $\parens{\ell+1,\ell,\ldots,\ell}$ or $\parens{\ell,\ell,\ldots,\ell,1}$}, and this single cell contains $n$. Since
  there are at least $4$ entries, $T$ does not have shape $\parens{2,1}$ and so the single cell containing $n$ is
  determined.
\end{proof}

\begin{lemma}
  \label{lem:1minors-deleting}
  Let $T \in \setyoung{n}$. Then the $1$-minors of $T - n$ can be obtained by deleting the entry $n-1$ from each of the
  $1$-minors of $T$.
\end{lemma}

\begin{proof}
  Note that $n-1$ must be in an outer corner of each of the $1$-minors of $T$ and in particular in $T - n$; thus
  deleting $n-1$ from any of these does not result in any sliding or renumbering.

  Let $m \leq n-1$ and consider deleting $m$ from $T$ and from $T - n$, giving $1$-minors $T - m$ and
  $\parens{T - n} - m$ respectively.

  Suppose first that the jeu de taquin process in $T$ ends at some entry other than $n$. Then the entry $n-1$ in $T - m$
  is obtained by renumbering $n$, which was not moved from its original location. The same jeu de taquin process occurs
  in $T - n$ and ends at the same entry. Hence $T - m$ and $\parens{T - n} - m$ differ only in the presence of $n-1$ in
  some outer corner of $T - m$, and so $\parens{T - m} - \parens{n-1} = \parens{T - n} - m$.

  Suppose now that the jeu de taquin process in $T$ ends at $n$. Then the last move of the jeu de taquin process in $T$
  is sliding the cell containing $n$ either vertically or horizontally; it is then renumbered to $n-1$. The jeu de
  taquin process in $T - n$ arising from deleting $m$ is the same except that it does not include this last move. Hence
  again $T - m$ and $\parens{T - n} - m$ differ only in the presence of $n-1$ in some outer corner of $T - m$, and so
  again $\parens{T - m} - \parens{n-1} = \parens{T - n} - m$.

  Thus the $1$-minor of $T - n$ formed by deleting $m$ equals the result of deleting $n-1$ from $T - m$.
  \textparens{Note in particular that $\parens{T - \parens{n-1}} - \parens{n-1} = \parens{T - n} - \parens{n-1}$ by
    taking $m = n-1$; hence this $1$-minor of $T - n$ arises by deleting $n-1$ from two different $1$-minors of $T$,
    namely $T-\parens{n-1}$ and $T-n$.}
\end{proof}

At this point, the results necessary for the induction step have been proved. It is now necessary to establish that all
tableaux with $5$ entries can be reconstructed from their sets of $1$-minors. It would be possible to do this by an
exhaustive calculation, but the following lemmata reduce the amount of calculation needed.

First, the following lemma is a consequence of \fullref{Lemma}{lem:1minor-shape-reconst} and the fact that there is
exactly one tableau with a single row or single column and a given number of entries.

\begin{lemma}
  \label{lem:1minor-row-or-column}
  A tableau in $\setyoung{n}$ for $n \geq 3$ with a single row \textparens{that is, shape $\parens{n}$} or a single
  column \textparens{that is, shape $\parens{1,1,\ldots,1}$} is reconstructible from its set of $1$-minors.
\end{lemma}

\begin{lemma}
  \label{lem:1minor-row-or-column-and-cell}
  A tableau in $\setyoung{n}$ for $n \geq 4$ with two rows \textparens{respectively, columns}, the second containing
  exactly one entry \textparens{that is, shape $\parens{n-1,1}$; respectively $\parens{2,1,1,\ldots,1}$} is
  reconstructible from its set of $1$-minors.
\end{lemma}

\begin{proof}
  The set of $1$-minors determines the shape of the tableau by \fullref{Lemma}{lem:1minor-shape-reconst}. Consider the
  case where the shape is $\parens{n-1,1}$; the other case is symmetrical. Note that it suffices to determine the entry
  in the single cell in the second row. By \fullref{Lemma}{lem:1minor-n-pos-reconst}, the location of the entry $n$ is
  determined by the $1$-minors. If it is the entry in the second row, the tableau is determined. So suppose it is not
  the entry in the second row; it must be at the right-hand end of the first row. Any deletion must affect the entry in
  the second row in one of three possible ways: the cell in the second row is either slid out of position, renumbered,
  or left untouched. And the last possibility arises at least once, when $n$ is deleted. Thus the entry in the second
  row will be the maximum entry that occurs in the second row of a $1$-minor. Hence the tableau is determined.
\end{proof}

The proof of the following lemma is the only piece of manual calculation necessary to establish the basis of the
induction.

\begin{lemma}
  \label{lem:1minor-32-or-221}
  Any tableau with shape $\parens{3,2}$ or $\parens{2,2,1}$ can be reconstructed from its set of $1$-minors.
\end{lemma}

\begin{proof}
  The set of $1$-minors determines the shape of the tableau by \fullref{Lemma}{lem:1minor-shape-reconst}. There are five
  tableaux with shape $\parens{3,2}$, and their sets of $1$-minors are distinct:
  \begin{align*}
    \setminors[\Bigg]{1}{\tableau{
    1 \& 2 \& 3 \\
    4 \& 5 \\
    }} &= \set[\Bigg]{
         \tableau{
         1 \& 2 \\
         3 \& 4 \\
         },\;
         \tableau{
         1 \& 2 \& 3 \\
         4 \\
         }
    }, \displaybreak[0]\\
    \setminors[\Bigg]{1}{\tableau{
    1 \& 2 \& 4 \\
    3 \& 5 \\
    }} &= \set[\Bigg]{
         \tableau{
         1 \& 3 \\
         2 \& 4 \\
         },\;
         \tableau{
         1 \& 2 \& 3 \\
         4 \\
         },\;
         \tableau{
         1 \& 2 \\
         3 \& 4 \\
         },\;
         \tableau{
         1 \& 2 \& 4 \\
         3 \\
         }
    }, \displaybreak[0]\\
    \setminors[\Bigg]{1}{\tableau{
    1 \& 3 \& 4 \\
    2 \& 5 \\
    }} &= \set[\Bigg]{
         \tableau{
         1 \& 2 \& 3 \\
         4 \\
         },\;
         \tableau{
         1 \& 3 \\
         2 \& 4 \\
         },\;
         \tableau{
         1 \& 3 \& 4 \\
         2 \\
         }
    }, \displaybreak[0]\\
    \setminors[\Bigg]{1}{\tableau{
    1 \& 2 \& 5 \\
    3 \& 4 \\
    }} &= \set[\Bigg]{
         \tableau{
         1 \& 3 \& 4 \\
         2 \\
         },\;
         \tableau{
         1 \& 2 \& 4 \\
         3 \\
         },\;
         \tableau{
         1 \& 2 \\
         3 \& 4 \\
         }
    }, \displaybreak[0]\\
    \setminors[\Bigg]{1}{\tableau{
    1 \& 3 \& 5 \\
    2 \& 4 \\
    }} &= \set[\Bigg]{
         \tableau{
         1 \& 2 \& 4 \\
         3 \\
         },\;
         \tableau{
         1 \& 3 \& 4 \\
         2 \\
         },\;
         \tableau{
         1 \& 3 \\
         2 \& 4 \\
         }
    }.
  \end{align*}
  Thus a tableau of shape $\parens{3,2}$ can be reconstructed from its set of $1$-minors. The case where the shape is
  $\parens{2,2,1}$ is symmetrical.
\end{proof}

\begin{theorem}
  \label{thm:set-reconst}
  Any tableau with at least $5$ entries can be reconstructed from its set of $1$-minors.
\end{theorem}

\begin{proof}
  Let $T \in \setyoung{n}$. First, the shape of the tableau $T$ can be reconstructed from its set of $1$-minors by
  \fullref{Lemma}{lem:1minor-shape-reconst}.

  Consider the following eight shapes: $\parens{5}$, $\parens{4,1}$, $\parens{3,2}$, $\parens{2,2,1}$,
  $\parens{2,1,1,1}$, $\parens{1,1,1,1,1}$ \textparens{all shapes of tableaux with $5$ entries except for
    $\parens{3,1,1}$}, and $\parens{3,1}$, $\parens{2,1,1}$ \textparens{which are shapes of tableaux with $4$ entries}.
  In each of these cases, the tableau can reconstructed from its set of $1$-minors by
  \fullref{Lemmata}{lem:1minor-row-or-column}, \ref{lem:1minor-row-or-column-and-cell}, or \ref{lem:1minor-32-or-221}.

  Now suppose that $n > 5$ or that $T$ has shape $\parens{3,1,1}$. The shape of the tableau $T$ can be reconstructed
  from its set of $1$-minors by \fullref{Lemma}{lem:1minor-shape-reconst}. The location of the entry $n$ in $T$ is
  determined by the set of $1$-minors by \fullref{Lemma}{lem:1minor-n-pos-reconst}. By
  \fullref{Lemma}{lem:1minors-deleting}, the set of $1$-minors of $T$ determines the set of $1$-minors of $T -
  n$. Therefore if $T - n$ is determined by its set of $1$-minors, then so is $T$, since the location of the entry $n$
  is determined.

  If $T$ has shape $\parens{3,1,1}$ then $T - n$ has shape $\parens{2,1,1}$ or $\parens{3,1}$ and so is determined by
  its set of $1$-minors by \fullref{Lemma}{lem:1minor-row-or-column-and-cell}. Combining this with the other shapes
  considered, one sees that any tableau with $5$ entries is determined by its set of $1$-minors.

  If $n > 5$, then by induction starting from $n = 5$, the tableau $T - n$ is indeed determined, and hence $T$ is also
  determined since the location of $n$ is determined.
\end{proof}

\section{%
  \texorpdfstring
  {Reconstruction from the multiset of $1$-minors}
  {Reconstruction from the multiset of 1-minors}%
}
\label{sec:reconst-multiset}

The following corollary is immediate from \fullref{Theorem}{thm:set-reconst}:

\begin{corollary}
  Any tableau with at least $5$ entries is reconstructible from its multiset of $1$-minors.
\end{corollary}

Among tableaux with $4$ entries, those of shapes $\parens{4}$, $\parens{3,1}$, $\parens{2,1,1}$ and $\parens{1,1,1,1}$
are reconstructible from their sets of $1$-minors by \fullref{Lemmata}{lem:1minor-row-or-column} and
\ref{lem:1minor-row-or-column-and-cell} and so also reconstructible from their multisets of $1$-minors.  The two
tableaux of shape $\parens{2,2}$ are not reconstructible from their multisets of $1$-minors:
\begin{align*}
\msetminors[\Bigg]{1}{ \tableau{ 1 \& 2 \\ 3 \& 4 \\ } }
&= \mset[\Bigg]{ \tableau{ 1 \& 3 \\ 2 \\ }, \tableau{ 1 \& 3 \\ 2 \\ }, \tableau{ 1 \& 2 \\ 3 \\ }, \tableau{ 1 \& 2 \\ 3 \\ } }, \\
\msetminors[\Bigg]{1}{ \tableau{ 1 \& 3 \\ 2 \& 4 \\ } }
&= \mset[\Bigg]{ \tableau{ 1 \& 2 \\ 3 \\ }, \tableau{ 1 \& 3 \\ 2 \\ }, \tableau{ 1 \& 2 \\ 3 \\ }, \tableau{ 1 \& 3 \\ 2 \\ } }.
\end{align*}

Among tableaux with $3$ entries, those of shapes $\parens{3}$ and $\parens{1,1,1}$ are reconstructible from their sets
of $1$-minors by \fullref{Lemma}{lem:1minor-row-or-column} and so also reconstructible from their multisets of
$1$-minors. The two tableaux of shape $\parens{2,1}$ are reconstructible from their multisets of $1$-minors but not
reconstructible from their sets of $1$-minors, since
\begin{align*}
\msetminors[\Bigg]{1}{ \tableau{ 1 \& 2 \\ 3 \\ } }
&= \mset[\Bigg]{ \tableau{ 1 \\ 2 \\ }, \tableau{ 1 \\ 2 \\ }, \tableau{ 1 \& 2 \\ } }, \\
\msetminors[\Bigg]{1}{ \tableau{ 1 \& 3 \\ 2 \\ } }
&= \mset[\Bigg]{ \tableau{ 1 \& 2 \\ }, \tableau{ 1 \\ 2 \\ }, \tableau{ 1 \& 2 \\ } },
\end{align*}
and so
\[
\setminors[\Bigg]{1}{ \tableau{ 1 \& 2 \\ 3 \\ } } =
\setminors[\Bigg]{1}{ \tableau{ 1 \& 3 \\ 2 \\ } } = \set[\Bigg]{ \tableau{ 1 \& 2 \\ }, \tableau{ 1 \\ 2 \\ } }.
\]

Not even the shape of a tableau with $2$ entries is determined by the set or multiset of its $1$-minors.

There is only one tableau with $1$ entry, so its reconstruction is trivial.

\section{Open problems}

The arguments above seem very specialized to the case of reconstruction from sets or multisets of $1$-minors.
The most immediate open problem is the following:

\begin{problem}
  For $k \geq 2$, for which $n$ is every tableau in $\setyoung{n}$ reconstructible from its set or multiset of
  $k$-minors?
\end{problem}

Reconstructibility of a tableau $T$ means that the multiset of its $k$-minors contains a sufficient amount of information for uniquely determining $T$.
Some of the information might be redundant; $T$ might be determined by just a few of its $k$-minors.
This raises the question what the minimum number of cards that guarantee reconstructibility is.

\begin{problem}
  % Does there exist a function $H_k(n)$ so that every tableau in $\setyoung{n}$ is uniquely determined by any
  % subset of $\setminors{k}{T}$ or
  % submultiset of $\msetminors{k}{T}$ of cardinality at least $H_k(n)$?
  Let $H_k(n)$ be the smallest number $m$ so that every tableau in $\setyoung{n}$ is uniquely determined by any
  submultiset of $\msetminors{k}{T}$ of cardinality at least $m$, provided that $\setyoung{n}$ is reconstructible.
  What are the numbers $H_k(n)$?
\end{problem}

A first step towards answering this question is the following result:

\begin{proposition}
  $H_1(n) \geq \floor{n/2} + 2$.
\end{proposition}

\begin{proof}
  It is sufficient to exhibit, for each $n$, two different tableaux of size $n$ with $\floor{n/2} + 1$ common
  $1$-minors. There are two cases, depending on the parity of $n$.

  For $n = 2k$, let
  \begin{align*}
    T_1 &= \tableau{
      |[doublewidth]| 1 \& 2 \& |[dashedentry]| \& k \& |[doublewidth]| k + 2 \& |[doublewidth]| k + 3 \& |[dashedentry]| \& 2k \\
      |[doublewidth]| k + 1 \\
    }\,, \\
    T_2 &= \tableau{
      1 \& 2 \& |[dashedentry]| \& |[doublewidth]| k-1 \& |[doublewidth]| k+1 \& |[doublewidth]| k+2 \& |[dashedentry]| \& 2k \\
      k \\
    }\,.
  \end{align*}
  The multisets of $1$-minors of $T_1$ and $T_2$ each contain $k = \floor{n/2}$ copies of
  \[
    \tableau{
      1 \& 2 \& |[dashedentry]| \& |[doublewidth]| k-1 \& |[doublewidth]| k+1 \& |[doublewidth]| k+2 \& |[dashedentry]| \& |[doublewidth]| 2k-1 \\
      k \\
    }\,,
  \]
  resulting from the deletion of $1,2,\ldots,k$ from $T_1$ and of $k+1,\ldots,2k$ from $T_2$, and $1$ copy of
  \[
    \tableau{ 1 \& 2 \& |[dashedentry]| \& |[doublewidth]| 2k-1 \\ }\,,
  \]
  resulting from the deletion of $k+1$ from $T_1$ and $k$ from $T_2$. Hence $T_1$ and $T_2$ are different tableaux of
  size $n$ with at least $\floor{n/2} + 1$ common $1$-minors.

  For $n = 2k+1$, let
  \begin{align*}
    T_3 &= \tableau{
      |[doublewidth]| 1 \& 2 \& |[dashedentry]| \& k \& |[doublewidth]| k + 2 \& |[doublewidth]| k + 3 \& |[dashedentry]| \& 2k \\
      |[doublewidth]| k + 1 \\
      |[doublewidth]| 2k + 1 \\
    }\,, \\
    T_4 &= \tableau{
      |[doublewidth]| 1 \& 2 \& |[dashedentry]| \& |[doublewidth]| k-1 \& |[doublewidth]| k+1 \& |[doublewidth]| k+2 \& |[dashedentry]| \& 2k \\
      |[doublewidth]| k \\
      |[doublewidth]| 2k + 1 \\
    }\,.
  \end{align*}
  The multisets of $1$-minors of $T_3$ and $T_4$ each contain $k = \floor{n/2}$ copies
  \[
    \tableau{
      1 \& 2 \& |[dashedentry]| \& |[doublewidth]| k-1 \& |[doublewidth]| k+1 \& |[doublewidth]| k+2 \& |[dashedentry]| \& |[doublewidth]| 2k-1 \\
      k \\
      2k\\
    }
  \]
  resulting from the deletion of $1,2,\ldots,k$ from $T_3$ and of $k+1,\ldots,2k$ from $T_4$, and $1$ copy of
  \[
    \tableau{
      1 \& 2 \& |[dashedentry]| \& |[doublewidth]| 2k-1 \\
      2k \\
    }\,.
  \]
  resulting from the deletion of $k+1$ from $T_3$ and of $k$ from $T_4$. Hence $T_3$ and $T_4$ are different tableaux of
  size $n$ with at least $\floor{n/2} + 1$ common $1$-minors.
\end{proof}

\bibliography{\jobname}

\begin{thebibliography}{Cam96}

\bibitem[Cam96]{Cameron_stories}
P.~J. Cameron.
\newblock `{S}tories from the age of reconstruction'.
\newblock {\em Congr. Numer.}, 113 (1996), pp. 31--41.
\newblock Festschrift for C. St. J. A. Nash-Williams.

\bibitem[Ful97]{fulton_young}
W.~Fulton.
\newblock {\em {Y}oung {T}ableaux: {W}ith {A}pplications to {R}epresentation
  {T}heory and {G}eometry}.
\newblock No.~35 in {\em LMS Student Texts}. Cambridge University Press, 1997.

\bibitem[Kel57]{Kelly}
P.~J. Kelly.
\newblock `{A} congruence theorem for trees'.
\newblock {\em Pacific J. Math.}, 7, no.~1 (1957), pp. 961--968.
\newblock {\sc url:}
  \href{https://projecteuclid.org:443/euclid.pjm/1103043674}{\nolinkurl{https://projecteuclid.org:443/euclid.pjm/1103043674}}.

\bibitem[Mnu93]{Mnukhin_properties}
V.~B. Mnukhin.
\newblock {\em {C}ombinatorial {P}roperties of {P}artially {O}rdered {S}ets and
  {G}roup {A}ctions}.
\newblock No.~8 in {\em Discrete Mathematics and Applications. TEMPUS Lecture
  Notes}. University of East Anglia, Norwich, 1993.

\bibitem[Mon09]{monks_solution}
M.~Monks.
\newblock `{T}he solution to the partition reconstruction problem'.
\newblock {\em J. Combin. Theory Ser. A}, 116, no.~1 (2009), pp. 76--91.
\newblock {\sc doi:} \href {http://dx.doi.org/10.1016/j.jcta.2007.12.012}
  {{10.1016/j.jcta.2007.12.012}}.

\bibitem[PS05]{PretzelSiemons}
O.~Pretzel \& J.~Siemons.
\newblock `{R}econstruction of partitions'.
\newblock {\em Electron. J. Combin.}, 11, no.~2 (2005), p. \#N5.
\newblock Festschrift for Richard Stanley.
\newblock {\sc doi:} \href {http://dx.doi.org/10.37236/1892} {{10.37236/1892}}.

\bibitem[Ula60]{Ulam}
S.~M. Ulam.
\newblock {\em {A} {C}ollection of {M}athematical {P}roblems}.
\newblock No.~8 in {\em Interscience Tracts in Pure and Applied Mathematics}.
  Interscience Publishers, New York, London, 1960.

\bibitem[Vat08]{Vatter}
V.~Vatter.
\newblock `{A} sharp bound for the reconstruction of partitions'.
\newblock {\em Electron. J. Combin.}, 15, no.~1 (2008), p. \#N23.
\newblock {\sc doi:} \href {http://dx.doi.org/10.37236/898} {{10.37236/898}}.

\end{thebibliography}
\bibliographystyle{alphaabbrv}

\end{document}